\documentclass[11pt]{article}
\usepackage{graphicx,indentfirst}
\usepackage{amscd}

\usepackage[all]{xy}
\usepackage{amsmath}
\usepackage{amsfonts,amssymb}
\usepackage{extarrows}
\usepackage{amsthm}
\usepackage{latexsym}
\usepackage[american]{babel}
\usepackage{times}
\usepackage{tikz}
\usepackage{tikz-cd}
\usepackage{url}
\usepackage{comment}
\usepackage[affil-it]{authblk}

\usetikzlibrary{matrix,arrows,decorations.pathmorphing}
\usepackage[a4paper,left=2.5cm,right=2.5cm,bottom=3cm,top=3.5cm]{geometry}

\makeatletter
\newcommand*{\relrelbarsep}{.386ex}
\newcommand*{\relrelbar}{%
  \mathrel{%
    \mathpalette\@relrelbar\relrelbarsep
  }%
}
\newcommand*{\@relrelbar}[2]{%
  \raise#2\hbox to 0pt{$\m@th#1\relbar$\hss}%
  \lower#2\hbox{$\m@th#1\relbar$}%
}
\providecommand*{\rightrightarrowsfill@}{%
  \arrowfill@\relrelbar\relrelbar\rightrightarrows
}
\providecommand*{\leftleftarrowsfill@}{%
  \arrowfill@\leftleftarrows\relrelbar\relrelbar
}
\providecommand*{\xrightrightarrows}[2][]{%
  \ext@arrow 0359\rightrightarrowsfill@{#1}{#2}%
}
\providecommand*{\xleftleftarrows}[2][]{%
  \ext@arrow 3095\leftleftarrowsfill@{#1}{#2}%
}

\makeatother

\RequirePackage[T1]{fontenc}

\begin{document}

\title{The descent of twisted perfect complexes   on a space with soft structure sheaf}

\author{Zhaoting Wei%
 \thanks{Email: \texttt{zwei3@kent.edu}}}
\affil{Kent State University at Geauga\\14111 Claridon Troy Road, Burton, OH 44021, USA}

\maketitle

\newtheorem{thm}{Theorem}[section]
\newtheorem{lemma}[thm]{Lemma}
\newtheorem{prop}[thm]{Proposition}
\newtheorem{coro}[thm]{Corollary}
\newtheorem{ques}[thm]{Question}
\newtheorem{conj}[thm]{Conjecture}
\theoremstyle{definition}\newtheorem{defi}{Definition}[section]
\theoremstyle{remark}\newtheorem{eg}{Example}
\theoremstyle{remark}\newtheorem{rmk}{Remark}
\theoremstyle{remark}\newtheorem{ctn}{Caution}
\renewcommand\appendixname{Appendix}


\begin{abstract}
In this paper we study the dg-category of twisted perfect complexes on a ringed space with soft structure sheaf. We prove that this dg-category is quasi-equivalent to the dg-category of complexes of vector bundles on that space. This result could be considered as a dg-enhancement of the classic result on soft sheaves in SGA6.
\end{abstract}

\section{Introduction}

The derived categories of perfect complexes  on ringed topoi were introduced in SGA 6 \cite{berthelot1966seminaire} in the 1960's. They have played an important role in mathematics ever since.

It is well-known that on a projective scheme $X$, any perfect complex of sheaves has a global resolution, i.e. it is quasi-isomorphic to a bounded complex of vector bundles. On the other hand, people also know that such resolutions do not exist on general ringed spaces $(X,\mathcal{A}_X)$. Nevertheless, if the structure sheaf $\mathcal{A}_X$ is \emph{soft}, global resolution for a perfect complex always exists. More precisely we have the following theorem.

 \begin{thm}[\cite{berthelot1966seminaire}, Expos\'{e} II, Proposition 2.3.2, See also \cite{block2010duality} Proposition 4.2]\label{thm: local perfect is global perfect}
 Suppose $(X, \mathcal{A}_X)$ is a ringed space, where $X$ is compact and $\mathcal{A}_X$ is a soft sheaf of rings. Then
 \begin{enumerate}
 \item The global sections functor
 $$
 \Gamma: \text{Mod}-\mathcal{A}_X\rightarrow \text{Mod}-\mathcal{A}_X(X)
 $$
 is exact and establishes and equivalence of categories between the category of sheaves of right $\mathcal{A}_X$-modules and the category of right modules over the global sections $\mathcal{A}_X(X)$.

 \item If $S\in \text{Mod}-\mathcal{A}_X$ locally has finite resolutions by finitely generated free $\mathcal{A}_X$-modules, then $\Gamma(X, S)$ has a finite resolution by finitely generated projectives.

 \item The derived category of perfect complexes of sheaves $D_{\text{perf}}(\text{Mod}-\mathcal{A}_X)$ is equivalent to the derived category of perfect complexes of modules $D_{\text{perf}}(\text{Mod}-\mathcal{A}_X(X))$.
 \end{enumerate}
  \end{thm}

On the other hand,  Toledo and Tong \cite{toledo1978duality} introduced  \emph{twisted complexes} in the 1970's as a way to obtain global resolutions of perfect complexes of sheaves on a complex manifold. In 2015 Wei proved in \cite{wei2016twisted} that the dg-category of twisted perfect complexes give a dg-enhancement of the derived category of perfect complexes.

  \begin{thm}[\cite{wei2016twisted}, Theorem 3.32]\label{thm: twisted complexes gives an dg-enhancement}
  Let $(X,\mathcal{A}_X)$ be a ringed space with soft structure ring and $\{U_i\}$ be any open cover of $X$. Let us denote the twisted perfect complex on the ringed space $(X, \mathcal{A}_X)$ which subject to the open cover $\{U_i\}$ by $\text{Tw}_{\text{perf}}(X, \mathcal{A}_X, U_i)$. Then $\text{Tw}_{\text{perf}}(X, \mathcal{A}_X, U_i)$  is a dg-enhancement of $D_{\text{perf}}(\text{Mod}-\mathcal{A}_X)$.
  \end{thm}

A detailed discussion of twisted complexes will be given in Section \ref{section: review of twisted complex}.

In this paper we study twisted perfect complexes on a ringed space with soft structure sheaf and we want to achieve a generalization of Theorem \ref{thm: local perfect is global perfect} to the dg-world. Let us denote the dg-category of bounded complexes of global sections of finitely generated locally free sheaves on $X$ by $\mathfrak{L}(X,\mathcal{A})$ and it is easy to show that $\mathfrak{L}(X,\mathcal{A})$ is an dg-enhancement of $D_{\text{perf}}(\text{Mod}-\mathcal{A}_X(X))$. In conclusion we have
$$
\begin{CD}
\text{Tw}_{\text{perf}}(X, \mathcal{A}_X, U_i)  @. \mathfrak{L}(X,\mathcal{A}_X)\\
@VV\text{dg-enhance}V @VV\text{dg-enhance}V\\
D_{\text{perf}}(\text{Mod}-\mathcal{A}_X) @>\sim>> D_{\text{perf}}(\text{Mod}-\mathcal{A}_X(X)).
\end{CD}
$$
Hence we expect that $\text{Tw}_{\text{perf}}(X, \mathcal{A}_X, U_i)$ and $\mathfrak{L}(X,\mathcal{A})$ are quasi-equivalent too. Actually this is the mean result of this paper.

   \begin{thm}[See Theorem \ref{thm: quasi-equivalence of bundles and twisted complexes} below]\label{thm: quasi-equivalence of bundles and twisted complexes in introduction}
 Let $X$ is a paracompact space with a soft structure sheaf $\mathcal{A}$. Moreover if the open cover $U_i$ is finite and good, then the dg-category $\mathfrak{L}(X,\mathcal{A})$ is quasi-equivalent to the dg-category of twisted complexes $\text{Tw}_{\text{perf}}(X, \mathcal{A}_X, U_i)$.
 \end{thm}

We would like to discuss the significance of Theorem \ref{thm: quasi-equivalence of bundles and twisted complexes} in descent theory. For an open cover $\{U_i\}$ of $X$, we consider the locally free sheaves on the \v{C}ech nerve and get the following cosimplicial diagram of dg-categories
\begin{equation}\label{equation: cosimplicial diagram of Cpx of an open cover}
\begin{tikzcd}
\prod \mathfrak{L}(U_i) \arrow[yshift=0.7ex]{r}\arrow[yshift=-0.7ex]{r}& \prod \mathfrak{L}(U_i\cap U_j) \arrow[yshift=1ex]{r}\arrow{r}\arrow[yshift=-1ex]{r}  &  \prod \mathfrak{L}(U_i\cap U_j\cap U_k)  \arrow[yshift=1.2ex]{r}\arrow[yshift=0.4ex]{r}\arrow[yshift=-0.4ex]{r}\arrow[yshift=-1.2ex]{r}&\cdots
\end{tikzcd}
\end{equation}

The descent data of $\mathfrak{L}(U_i)$ is given by the \emph{homotopy limit} of Diagram (\ref{equation: cosimplicial diagram of Cpx of an open cover}). Recently in \cite{block2017explicit} it has been proved that $\text{Tw}_{\text{perf}}(X, \mathcal{A}_X, U_i)$ is quasi-equivalent to the homotopy limit of  (\ref{equation: cosimplicial diagram of Cpx of an open cover}). Hence the result in Theorem \ref{thm: quasi-equivalence of bundles and twisted complexes} can be reinterpreted as
   \begin{thm}[See Theorem \ref{thm: locally free sheaves is weakly equivalent to the homotopy limit in section} below]\label{thm: locally free sheaves is weakly equivalent to the homotopy limit}
 Let $X$ is a paracompact space with a soft structure sheaf $\mathcal{A}$. Moreover if the open cover $U_i$ is finite and good, then we have a quasi-equivalence between dg-categories
 $$
 \mathfrak{L}(X)\overset{\sim}{\to}\text{holim}_{I}\mathfrak{L}(U_I)
 $$
 \end{thm}

This paper is organized as follows. In Section \ref{section: review of twisted complex} we quickly review the theory of twisted complexes. In Section \ref{section: twisted functor is a quasi-equivalence} we prove that $\mathfrak{L}(X,\mathcal{A})$ is quasi-equivalent to  $\text{Tw}_{\text{perf}}(X, \mathcal{A}_X, U_i)$. In more details, in Section \ref{subsection: twisted functor} we define the dg-functor $\mathcal{T}:\mathfrak{L}(X,\mathcal{A})\to \text{Tw}_{\text{perf}}(X, \mathcal{A}_X, U_i)$, in Section \ref{subsection: quasi-essentially surjective} we prove $\mathcal{T}$ is quasi-essentially surjective and in section \ref{subsection: quasi-fully faithful} we prove $\mathcal{T}$ is quasi-fully faithful. In Section \ref{section: twisted complex and homotopy limit} we discuss the relation between twisted perfect complexes and descent theory.

In Appendix \ref{appendix: some generalities of soft sheaves} we collect some general results on soft sheaves which are necessary in this paper.

\begin{rmk}
Part of this paper, in particular Section \ref{subsection: quasi-essentially surjective}, has been integrated in to \cite{wei2018descent}, although the notations and viewpoint are slightly different.
\end{rmk}

\section{A quick review of twisted perfect complexes}\label{section: review of twisted complex}
We give a brief review of twisted perfect complexes in this subsection, for reference see \cite{o1981trace} Section 1 or \cite{wei2016twisted} Section 2.

Let $(X,\mathcal{A})$ be a ringed paracompact topological space and $\mathcal{U}=\{U_i\}$ be an locally finite open cover of $X$. Let $U_{i_0\ldots i_n}$ denote the intersection $U_{i_0}\bigcap\ldots \bigcap U_{i_n}$.

For each $U_{i_k}$, let $E^{\bullet}_{i_k}=\{E^{r}_{i_k}\}_{r\in \mathbb{Z}}$ be a two-side bounded graded sheaf of $\mathcal{A}$-modules on $U_{i_k}$. Let
$$
C^{\bullet}(\mathcal{U},E^{\bullet})=\bigoplus_{p,q}C^p(\mathcal{U},E^q)
$$
be the bigraded complexes of $E^{\bullet}$. More precisely, an element $c^{p,q}$ of $C^p(\mathcal{U},E^q)$ consists of a section $c^{p,q}_{i_0\ldots i_p}$ of $E^{q}_{i_0}$ over each non-empty intersection $U_{i_0\ldots i_n}$. If $U_{i_0\ldots i_p}=\emptyset$, simply let the component on it be zero.

Now if another two-side bounded graded sheaf $F^{\bullet}_{i_k}$ of $\mathcal{A}$-modules is given on each $U_{i_k}$, then we can consider the map
\begin{equation}\label{equation: map with bigrade between graded sheaves}
C^{\bullet}(\mathcal{U},\text{Hom}^{\bullet}(E,F))=\bigoplus_{p,q}C^p(\mathcal{U},\text{Hom}^q(E,F)).
\end{equation}
An element $u^{p,q}$ of $C^p(\mathcal{U},\text{Hom}^q(E,F))$ gives a section $u^{p,q}_{i_0\ldots i_p}$ of $\text{Hom}^q_{\mathcal{A}-\text{Mod}}(E_{i_p},F_{i_0})$ over each non-empty intersection $U_{i_0\ldots i_n}$. Notice that we require $u^{p,q}$ to be a map from the $E$ on the last subscript of $U_{i_0\ldots i_n}$ to the $F$ on the first subscript of $U_{i_0\ldots i_n}$.

We need to study the compositions of $C^{\bullet}(\mathcal{U},\text{Hom}^{\bullet}(E,F))$. Let $\{G^{\bullet}_{i_k}\}$ be a third two-side bounded graded sheaf of $\mathcal{A}$-modules. There is a composition map
$$
C^{\bullet}(\mathcal{U},\text{Hom}^{\bullet}(F,G)) \times C^{\bullet}(\mathcal{U},\text{Hom}^{\bullet}(E,F))\rightarrow C^{\bullet}(\mathcal{U},\text{Hom}^{\bullet}(E,G)).
$$
In fact, for $u^{p,q}\in C^p(\mathcal{U},\text{Hom}^q(F,G))$ and $v^{r,s} \in C^{r}(\mathcal{U},\text{Hom}^{s}(E,F))$, their composition $(u\cdot v)^{p+r,q+s}$ is given by (see \cite{o1981trace} Equation (1.1))
\begin{equation}\label{equation: composition of maps between graded sheaves}
(u\cdot v)^{p+r,q+s}_{i_0\ldots i_{p+r}}=(-1)^{qr}u^{p,q}_{i_0\ldots i_p}v^{r,s}_{i_p\ldots i_{p+r}}
\end{equation}
where the right hand side is the composition of sheaf maps.

In particular $C^{\bullet}(\mathcal{U},\text{Hom}^{\bullet}(E,E))$ becomes an associative algebra under this composition (It is easy but tedious to check the associativity). We also notice that $C^{\bullet}(\mathcal{U},E^{\bullet})$ becomes a left module over this algebra. In fact the action
$$
C^{\bullet}(\mathcal{U},\text{Hom}^{\bullet}(E,E))\times C^{\bullet}(\mathcal{U},E^{\bullet})\rightarrow  C^{\bullet}(\mathcal{U},E^{\bullet})
$$
is given by $(u^{p,q},c^{r,s})\mapsto (u\cdot c)^{p+r,q+s}$ where (see \cite{o1981trace} Equation (1.2))
\begin{equation}\label{equation: action of maps on sheaves}
(u\cdot c)^{p+r,q+s}_{i_0\ldots i_{p+r}}=(-1)^{qr}u^{p,q}_{i_0\ldots i_p}c^{r,s}_{i_p\ldots i_{p+r}}
\end{equation}
where the right hand side is given by evaluation.

There is also a \v{C}ech-style differential operator $\delta$ on $C^{\bullet}(\mathcal{U},\text{Hom}^{\bullet}(E,F))$ and $C^{\bullet}(\mathcal{U},E^{\bullet})$ of bidegree $(1,0)$ given by the formula
\begin{equation}\label{equation: delta on maps}
(\delta u)^{p+1,q}_{i_0\ldots i_{p+1}}=\sum_{k=1}^p(-1)^k u^{p,q}_{i_0\ldots \widehat{i_k} \ldots i_{p+1}}|_{U_{i_0\ldots i_{p+1}}} \,\text{ for } u^{p,q}\in C^p(\mathcal{U},\text{Hom}^q(E,F))
\end{equation}
and
\begin{equation}\label{equation: delta on sheaves}
(\delta c)^{p+1,q}_{i_0\ldots i_{p+1}}=\sum_{k=1}^{p+1}(-1)^k c^{p,q}_{i_0\ldots \widehat{i_k} \ldots i_{p+1}}|_{U_{i_0\ldots i_{p+1}}} \,\text{ for }c^{p,q}\in C^p(\mathcal{U},E).
\end{equation}
It is not difficult to check that the \v{C}ech differential satisfies the Leibniz rule.

Now we can introduce the definition of twisted perfect complex.

\begin{defi}\label{defi: twisted complex}
Let $(X,\mathcal{A})$ be a ringed paracompact topological space and $\mathcal{U}=\{U_i\}$ be an locally finite open cover of $X$. A \emph{twisted perfect complex} consists of locally free graded sheaves $E^{\bullet}_{i}$ of $\mathcal{A}$-modules together with a collection of morphisms
$$
a=\sum_{k \geq 0} a^{k,1-k}
$$
where $a^{k,1-k}\in C^k(\mathcal{U},\text{Hom}^{1-k}(E,E)) $ and which satisfies the equation
\begin{equation}\label{equation: MC for twisted complex}
\delta a+ a\cdot a=0.
\end{equation}
More explicitly, for $k\geq 0$
\begin{equation}\label{equation: MC for twisted complex explicit}
\delta a^{k-1,2-k}+ \sum_{i=0}^k a^{i,1-i}\cdot a^{k-i,1-k+i}=0.
\end{equation}

The twisted perfect complexes on $(X,\mathcal{A}, \{U_i\})$ form a dg-category: the objects are the twisted perfect complexes $(E^{\bullet}_i,a)$ and the morphism from $\mathcal{E}=(E^{\bullet}_i,a)$ to $\mathcal{F}=(F^{\bullet}_i,b)$ are $C^{\bullet}(\mathcal{U},\text{Hom}^{\bullet}(E,F))$ with the total degree. Moreover, the differential on a morphism $\phi$ is given by
\begin{equation}\label{equation: differential on morphisms of twisted complexes}
d \phi=\delta \phi+b\cdot \phi-(-1)^{|\phi|}\phi\cdot a.
\end{equation}

We denote the dg-category of twisted complexes on $(X,\mathcal{A}, \{U_i\})$ by $\text{Tw}_{\text{perf}}(X, \mathcal{A}, U_i)$. If there is no danger of confusion we can simply denote it by $\text{Tw}_{\text{perf}}(X)$.
\end{defi}

\begin{rmk}
Twisted perfect complex is called \emph{twisted cochain} in \cite{o1981trace}. We use the terminology "twisted perfect complex" to emphasize that each $E^{\bullet}_i$ is a strictly perfect complex.
\end{rmk}

In \cite{wei2016twisted} Section 2.5 it has been shown that $\text{Tw}_{\text{perf}}(X, \mathcal{A}, U_i)$ has a pre-triangulated structure hence Ho$\text{Tw}_{\text{perf}}(X, \mathcal{A}, U_i)$ is a triangulated category.  In more details we have the following definitions.

\begin{defi}\label{defi: shift of twisted complex}[Shift]
Let $\mathcal{E}=(E^{\bullet}_i,a)$ be a twisted perfect complex. We define its shift   $\mathcal{E}[1]$ to be $\mathcal{E}[1]=(E[1]^{\bullet}_i,a[1])$ where
$$
E[1]^{\bullet}_i=E^{\bullet+1}_i \text{ and } a[1]^{k,1-k}=(-1)^{k-1}a^{k,1-k}.
$$

Moreover, let $\phi: \mathcal{E}\to \mathcal{F}$ be a morphism. We define its shift $\phi[1]$ as
$$
\phi[1]^{p,q}=(-1)^q\phi^{p,q}.
$$
\end{defi}

\begin{defi}\label{defi: mapping cone}[Mapping cone]
Let $\phi^{\bullet,-\bullet}$ be a closed degree zero map between twisted perfect complexes $\mathcal{E}=(E^{\bullet},a^{\bullet,1-\bullet})$ and $\mathcal{F}=(F^{\bullet},b^{\bullet,1-\bullet})$ , we can define the \emph{mapping cone} $\mathcal{G}=(G,c)$ of $\phi$ as follows (see \cite{o1985grothendieck} Section 1.1):
$$
G^n_i:=E^{n+1}_i\oplus F^n_i
$$
and
\begin{equation}\label{equation: diff in mapping cone}
c^{k,1-k}_{i_0\ldots i_k}=\begin{pmatrix}(-1)^{k-1} a^{k,1-k}_{i_0\ldots i_k}&0\\ (-1)^k\phi^{k,-k}_{i_0\ldots i_k}&b^{k,1-k}_{i_0\ldots i_k}\end{pmatrix}.
\end{equation}
\end{defi}

We have some further results on the morphisms between twisted perfect complexes.

\begin{prop}[\cite{wei2016twisted} Proposition 2.31]\label{prop: homotopy invertible morphisms}
Let the cover $\{U_i\}$  satisfies   $H^k(U_i,\mathcal{F})=0$ for any $i$, any  sheaf $\mathcal{F}$ on $U_i$ and any $k\geq 1$. If $\mathcal{E}$ and $\mathcal{F}$ are both in the subcategory $\text{Tw}_{\text{perf}}(X)$, then a closed degree zero morphism $\phi$ between twisted complexes $\mathcal{E}$ and $\mathcal{F}$ is invertible in the homotopy category $\text{HoTw}_{\text{perf}}(X)$   if and only if its $(0,0)$ component
$$
\phi^{0,0}_i: (E^{\bullet}_i,a^{0,1}_i)\to (F^{\bullet}_i,b^{0,1}_i)
$$
is a quasi-isomorphism of complexes of $\mathcal{A}$-modules on $U_i$ for each $i$. In this case we call $\phi$ a \emph{weak equivalence} between  $\mathcal{E}$ and $\mathcal{F}$.
\end{prop}
\begin{proof}
See \cite{wei2016twisted} Proposition 2.31.
\end{proof}

As we mentioned before, the dg-category of twisted perfect complexes gives a dg-enhancement of $D_{\text{perf}}(\mathcal{A}-\text{mod})$. To state our theorem  we first introduce the concept of p-good cover.

\begin{defi}\label{defi: good space}
A locally ringed space $(U,\mathcal{O}_U)$ is called \emph{p-good} if it satisfies the following two conditions
\begin{enumerate}
\item For every perfect complex $\mathcal{P}^{\bullet}$ on $U$, there exists a strictly perfect complex $\mathcal{E}^{\bullet}$ on $U$ together with a quasi-isomorphism $u: \mathcal{E}^{\bullet}\overset{\sim}{\to}\mathcal{P}^{\bullet}$.
\item The higher cohomologies of any sheaf vanish, i.e. $H^k(U,\mathcal{F})=0$ for any   sheaf $\mathcal{F}$ on $U$ and any $k\geq 1$.
\end{enumerate}

An open cover $\{U_i\}$ of $X$ is called a \emph{p-good cover} if $(U_I,\mathcal{O}_X|_{U_I})$ is a p-good space for   any finite intersection $U_I$ of the open cover.
\end{defi}

\begin{thm}[\cite{wei2016twisted} Theorem 3.32]\label{thm: twisted complexes gives an dg-enhancement in general}
For a p-good cover $\{U_i\}$ of $X$ we have an equivalence of categories
\begin{equation}
\mathcal{S}: Ho\text{Tw}_{\text{perf}}(X)\overset{\sim}{\to} D_{\text{perf}}(\mathcal{A}-\text{mod}).
\end{equation}
In other words, $\text{Tw}_{\text{perf}}(X, \mathcal{A}, U_i)$ gives a dg-enhancement of $D_{\text{perf}}(\mathcal{A}-\text{mod})$.
\end{thm}

\begin{rmk}
Actually if the structure sheaf $\mathcal{A}$ is soft, it is easy to show that any open cover $\{U_i\}$ is p-good hence $\text{Tw}_{\text{perf}}(X, \mathcal{A}, U_i)$ gives a dg-enhancement of $D_{\text{perf}}(\mathcal{A}-\text{mod})$.
\end{rmk}

If $\mathcal{A}$ is soft, we can even take $\{U_i\}$ to be one single open set $X$ and Theorem \ref{thm: twisted complexes gives an dg-enhancement in general} still holds. In this case $\text{Tw}_{\text{perf}}(X, \mathcal{A}, U_i)$  reduces to  $\mathfrak{L}(X,\mathcal{A})$,  the dg-category of bounded complexes of global sections of locally free finitely generated sheaves on $X$, and we see that $\mathfrak{L}(X,\mathcal{A})$ is a dg-enhancement of $D_{\text{perf}}(\mathcal{A}-\text{mod})$.

\section{The descent of twisted perfect complexes}\label{section: twisted functor is a quasi-equivalence}
\subsection{The twisting functor $\mathcal{T}$}\label{subsection: twisted functor}
In this subsection we state the main theorem in this paper, i.e. $\mathfrak{L}(X,\mathcal{A})$ is quasi-equivalent to $\text{Tw}_{\text{perf}}(X, \mathcal{A}, U_i)$.

First we define a natural dg-functor from $\mathfrak{L}(X,\mathcal{A})$  to $\text{Tw}_{\text{perf}}(X, \mathcal{A}, U_i)$ as follows
\begin{defi}[\cite{wei2016twisted} Definition 3.11]\label{defi: twisted functor}
Let $(S^{\bullet},d)$ be a complex of  $\mathcal{A}(X)$-modules. We define its associated twisted perfect complex, $\mathcal{T}(S)$,  by restricting to the $U_i$'s. In more details let $(E^{\bullet},a)=\mathcal{T}(S)$ then
$$
E^{n}_i=S^n|_{U_i}
$$
and
$$
a^{0,1}_i=d|_{U_i},~a^{1,0}_{ij}=id \text{ and } a^{k,1-k}=0 \text{ for }k\geq 2.
$$
The $\mathcal{T}$ of morphisms is defined in a similar way.

We call the dg-functor $\mathcal{T}: \mathfrak{L}(X,\mathcal{A})\to \text{Tw}_{\text{perf}}(X, \mathcal{A}, U_i)$ the \emph{twisting functor}.
\end{defi}

We want to prove that the  $\mathcal{T}$ is actually a quasi-equivalence between dg-categories. First we need some condition on the open cover $\{U_i\}$.

\begin{defi}\label{defi: good cover}
A open cover $U_i$ of $X$ is good if any finite intersection $U_{i_0\ldots i_k}$ is contractible.
\end{defi}

It is well-known that good open cover exists for a wide ranged of spaces $X$. With Definition \ref{defi: good cover} we could state the main theorem of this paper.

\begin{thm}\label{thm: quasi-equivalence of bundles and twisted complexes}
 Let $X$ is a paracompact space with a \emph{soft} structure sheaf $\mathcal{A}$. Moreover if the open cover $U_i$ is finite and good, then the dg-functor $\mathcal{T}: \mathfrak{L}(X,\mathcal{A})\to \text{Tw}_{\text{perf}}(X, \mathcal{A}, U_i)$ is a quasi-equivalence.
 \end{thm}

 From now on we always assume that the structure sheaf $\mathcal{A}$  is soft. We leave the discussion of the properties of soft sheaves in Appendix \ref{appendix: some generalities of soft sheaves}.

\subsection{The quasi-essential surjectivity of $\mathcal{T}$}\label{subsection: quasi-essentially surjective}
First we prove that $\mathcal{T}$ is  quasi-essentially surjective.

 \begin{prop}[Quasi-essential surjectivity]\label{prop: descent of smooth complex on Cech cover}
 Let $X$ is a paracompact space with a soft structure sheaf $\mathcal{A}$.  If the open cover $\{U_i\}$ is finite and good, then for every twisted perfect complex $\mathcal{F}=(F_i, b)$, there is a bounded complex of finitely generated locally free $\mathcal{A}$-modules $\mathcal{E}=(E^{\bullet},d)$ such that $\mathcal{F}$ is isomorphic to $\mathcal{T}(\mathcal{E})$ in the homotopy category HoTw$(X)$.

 In more details, there exists a twisted complex $E=(E^{\bullet}, a)$ such that $a^{0,1}=d$, $a^{1,0}=id$ and $a^{k,1-k}=0$ for all $k\geq 2$ together with a  degree zero morphism
 $\phi: E\rightarrow F$ such that  for all $k \geq 0$ and all $U_{i_0\ldots i_k}$ we have a $\mathcal{A}$-module map
 $$
 \phi^{k,-k}_{i_0\ldots i_k}: E^{\bullet}|_{U_{i_0\ldots i_k}}\rightarrow F_{i_0}^{\bullet-k}|_{U_{i_0\ldots i_k}}
 $$
which satisfies the following two conditions
 \begin{enumerate}
 \item The $\phi^{k,-k}$'s intertwine $a$ and $b$:
 \begin{equation}\label{equation: phi intertwine a and b}
( \delta \phi^{k-1,1-k})_{i_0\ldots i_k}+\sum_{p=0}^k b^{p,1-p}_{i_0\ldots i_p}\phi^{k-p,p-k}_{i_p\ldots i_k}-\phi^{k,-k}_{i_0\ldots i_k} a^{0,1}_{i_k}-\phi^{k-1,1-k}_{i_0\ldots i_{k-1}}|_{U_{i_0\ldots i_k}}=0.
 \end{equation}
\item $\phi$ is invertible in the homotopy category HoTw$(X)$. According to Proposition \ref{prop: homotopy invertible morphisms}, this is the same to say that for each $i$, the map $\phi^{0,0}_i$ induces a quasi-isomorphism from $(E^{\bullet}|_{U_i}, a^{0,1}_i)$ to $(F^{\bullet}_i, b^{0,1}_i)$.
 \end{enumerate}
 \end{prop}

The rest of this subsection devotes to the proof of Proposition \ref{prop: descent of smooth complex on Cech cover}. First we introduce the following definitions.

\begin{defi}\label{defi: descent data module Q}
Let $P_i,~Q_i$ be  sheaves of $\mathcal{A}$-modules on $U_i$. A  \emph{descent data of $P_i$ modulo $Q_i$} consists of the following data
\begin{itemize}
\item a collection of $\mathcal{A}$-module maps
$$
\tau_i: Q_i\rightarrow P_i;
$$

\item a collection of $\mathcal{A}$-module maps
$$
\theta_{ji}: P_i|_{U_{ji}}\rightarrow P_j|_{U_{ji}};
$$
\item a collection of $\mathcal{A}$-module maps
$$
\vartheta_{kji}: P_i|_{U_{kji}}\rightarrow Q_k|_{U_{kji}};
$$
\end{itemize}
which satisfy the following conditions
\begin{enumerate}
\item  $\theta_{ji}$ is a cocycle module $Q_i$, more precisely
\begin{equation}\label{equation: transittion function modulo Q}
\theta_{ki}-\theta_{kj}\theta_{ji}=\tau_k \vartheta_{kji};
\end{equation}

\item $
\theta_{ii}=id_{P_i}$.
\end{enumerate}
\end{defi}

\begin{defi}\label{defi: descent of P to X module Q}
Let $P_i,~Q_i$ be  sheaves of $\mathcal{A}$-modules on $U_i$. A  \emph{descent of $P_i$ to $X$ modulo $Q_i$} consists of a   sheaf of $\mathcal{A}$-modules $R$ on $X$ together with the following data
\begin{itemize}
\item a collection of $\mathcal{A}$-module maps
$$
\tau_i: Q_i\rightarrow P_i;
$$

\item a collection of $\mathcal{A}$-module maps
$$
\psi_i: R|_{U_i}\rightarrow P_i
$$
\item a collection of $\mathcal{A}$-module maps
$$\xi_{ji}: R|_{U_{ji}}\rightarrow Q_j|_{U_{ji}};$$
\end{itemize}
which satisfy the following conditions
\begin{enumerate}
\item  On $U_{ji}$ the maps $\psi_i$ and $\psi_j$ are compatible with the transition function $\theta_{ji}$ modulo $Q_i$, i.e.
$$
\psi_j-\theta_{ji}\psi_i=\tau_j \xi_{ji}
$$

\item Each of the $\psi_i$ is surjective modulo $Q_i$ in the following sense: For any point $x\in U_i$ and $u_i$ in the fiber $P_i|_x$, there exist an $v\in R|_x$ and a $w_i\in Q_i|_x$, such that
$$
u_i=\psi_i(v)-\tau_i(w_i).
$$
\end{enumerate}
\end{defi}

We have the following lemma on the existence of descent modulo $Q_i$.
\begin{lemma}\label{lemma: gluing bundles by transition function modulo Q}
Let $P_i,~Q_i$ be locally free finitely generated sheaves of $\mathcal{A}$-modules on $U_i$ with a descent data of $P_i$ modulo $Q_i$.
Then there exists a locally free finitely generated sheaf of $\mathcal{A}$-modules $R$ on $X$ which is a descent of $P_i$ to $X$ modulo $Q_i$.
\end{lemma}
\begin{proof}
First we construct the sheaf $R$. By Lemma \ref{lemma: extension of a locally free finitely generated sheaf} we can extend each $P_i$ from $U_i$ to a locally free finitely generated sheaf of $\mathcal{A}$-modules on $X$. Let's denote the extension  by $\widetilde{P_i}$. By Proposition \ref{prop: sheaf of modules of soft rings is fine}, $\widetilde{P_i}$ is still fine. Then we define
$$
R=\bigoplus_i \widetilde{P_i}
$$
as a sheaf of $\mathcal{A}$-module on $X$. Since the cover $U_i$ is finite, $R$ is still finitely generated.

Next we construct the maps $\psi_i: R|_{U_i}\rightarrow P_i$. By the definition of $R$, a section $r$ of $R$ is of the form
$$
r=(\widetilde{p_1},\ldots,\widetilde{p_n})
$$
where the $\widetilde{p_j}$'s are  sections of $\widetilde{P_j}$. Then $\widetilde{p_j}|_{U_j}$ is a section of $P_j$. The naive way to define $\psi_i$ is to apply the transition functions $\theta_{ij}$ directly on $\widetilde{p_j}$ and sum them up. The problem is that $\theta_{ij}\widetilde{p_j}$ is defined only on $U_{ij}$ instead of  $U_i$ hence we cannot get a well-defined map

To solve this problem we use Lemma \ref{lemma: partition of unity} to obtain a partition of unity $\{\rho_j\}$ and  we notice that $
\theta_{ij}(\rho_j\widetilde{p_j})=\rho_j\theta_{ij}(\widetilde{p_j})$
 can be extended by $0$ from $U_{ij}$ to $U_i$. Then we define $\psi_i$ by
\begin{equation}\label{equation: construction of psi}
\begin{split}
\psi_i: R|_{U_i}& \rightarrow P_i\\
(\widetilde{p_1},\ldots,\widetilde{p_n})&\mapsto \sum_j \theta_{ij}(\rho_j\widetilde{p_j}).
\end{split}
\end{equation}

It is obvious that $\psi_i$ is a smooth bundle map. Now we need to prove that $\psi_i$ has the required properties.

First we prove that  $\psi_i$  is surjective modulo $Q_i$. For any point $x\in U_i$ and $u_i$ in the fiber $P_i|_x$ we just define the vectors $v_j\in P_j|_x$ as
$$
v_j=\begin{cases}\theta_{ji} (u_i) &\text{ if } x\in U_j \\
0 &\text{ if } x \not\in U_j
\end{cases}
$$
and $v\in R|_x$ as $v=(v_1,\ldots v_n)$. Then
$$
\psi_i(v)=\sum_j\rho_j \theta_{ij}(v_j)
=\sum_j\rho_j \theta_{ij}\theta_{ji} (u_i).
$$

Now we use the "cocycle modulo $Q_i$" condition on $\theta_{ij}$'s. By Equation (\ref{equation: transittion function modulo Q})  we know that
$$
\theta_{ij}\theta_{ji} =\theta_{ii}-\tau_i\vartheta_{iji}=id-\tau_i\vartheta_{iji},
$$
hence
\begin{equation*}
\begin{split}
\psi_i(v)=&\sum_j \rho_j u_i-\sum_j\rho_j \tau_i\vartheta_{iji}(u_i)\\
=& u_i-\tau_i(\sum_j\rho_j \vartheta_{iji}(u_i)).
\end{split}
\end{equation*}
Take $w_i=\sum_j\rho_j \vartheta_{iji}(u_i)\in Q_i|_x$ we have proved the surjectivity of $\psi_i$.

The proof of the compatibility of $\psi_i$ and $\psi_j$ is similar. We define the map $\xi_{ji}: R|_{U_{ji}}\rightarrow Q_j|_{U_{ji}}$ as follows: for any $r=(\widetilde{p_1},\ldots,\widetilde{p_n})$ a section of $R|_{U_{ji}}$ define
$$
\xi_{ji}(\widetilde{p_1},\ldots,\widetilde{p_n}):=\sum_k \vartheta_{jik}(\rho_k\widetilde{p_k})=\sum_k \rho_k\vartheta_{jik}(\widetilde{p_k}).
$$
Then
$$
\theta_{ji}\psi_i(\widetilde{p_1},\ldots,\widetilde{p_n})=\theta_{ji} \sum_k \rho_k\theta_{ik}\widetilde{p_k}=\sum_k \rho_k\theta_{ji}\theta_{ik}\widetilde{p_k}.
$$
Again by Equation (\ref{equation: transittion function modulo Q}) we have
$$
\theta_{ji}\theta_{ik}=\theta_{jk}-\tau_j \vartheta_{jik}
$$
hence
\begin{equation*}
\begin{split}
\theta_{ji}\psi_i(\widetilde{p_1},\ldots,\widetilde{p_n})=&\sum_k \rho_k(\theta_{jk}\widetilde{p_k}-\tau_j \vartheta_{jik}\widetilde{p_k})\\
=& \sum_k \rho_k\theta_{jk}\widetilde{p_k}-\sum_k \rho_k\tau_j \vartheta_{jik}\widetilde{p_k}\\
=& \sum_k \rho_k\theta_{jk}\widetilde{p_k}-\tau_j(\sum_k \rho_k\vartheta_{jik}\widetilde{p_k})\\
=& \psi_j (\widetilde{p_1},\ldots,\widetilde{p_n})-\tau_j \xi_{ji}(\widetilde{p_1},\ldots,\widetilde{p_n}).
\end{split}
\end{equation*}
hence we get the compatibility of $\psi_i$ and $\psi_j$ .
\end{proof}

\begin{rmk}
Lemma \ref{lemma: gluing bundles by transition function modulo Q} requires $X$ to be compact. More precisely it requires the cover $U_i$ to be finite. The sheaf $R$ constructed in the proof has a very large rank over $\mathcal{A}$ but it is still locally free and finitely generated.
\end{rmk}

To prove Proposition \ref{prop: descent of smooth complex on Cech cover}, we need the following lemma, which is a variation of Lemma \ref{lemma: gluing bundles by transition function modulo Q}.
\begin{lemma}\label{lemma: higher gluing bundles by twisted cochains}
Let $(F^{\bullet}_i,b)$ be a twisted perfect complex for $U_i$. If the number $l$ has the property that for any $n>l$ and any $i$, the cohomology
$H^n(F^{\bullet}_i, b^{0,1}_i)$ vanishes, then there exists a locally free finitely generated sheaf of $\mathcal{A}$-modules $E^l$ on $X$ together with  $\mathcal{A}$-module maps
$$
\psi^{k,-k}: E^l|_{U_{i_0\ldots i_k}}\rightarrow F^{l-k}_{i_0}|_{U_{i_0\ldots i_k}}
$$
which satisfy the following two conditions:
\begin{enumerate}\label{equation: Maurer-Cartan in lemma}
\item For each $i$, the map $\psi^{0,0}: E^l|_{U_i}\rightarrow F^l_i$ has its image in $\ker b^{0,1}_i\subset F^l_i$ and is surjective  modulo $F^{l-1}_i$ in the same sense as Condition 1 in Lemma \ref{lemma: gluing bundles by transition function modulo Q}.

\item On each $U_{i_0\ldots i_k}$ the $\psi$'s satisfy the equation
\begin{equation}
\delta \psi^{k-1,1-k}+(-1)^{k}\psi^{k-1,1-k}_{i_0\ldots i_{k-1}}|_{U_{i_0\ldots i_{k}}}+\sum_{p=0}^k b^{p,1-p}\,\psi^{k-p,p-k}=0
\end{equation}

\end{enumerate}
\end{lemma}
\begin{proof}
First by Corollary \ref{coro: kernel of acyclic above complex is a vector bundle} we know that $\ker b^{0,1}_i$ is a locally free finitely generated sub-sheaf of $F^l_i$. Then similar to the proof of  Lemma \ref{lemma: gluing bundles by transition function modulo Q}, let's first extend $\ker b^{0,1}_i$ on $U_i$ to $\widetilde{\ker b^{0,1}_i}$ on $X$ and define
\begin{equation}\label{equation: construction of E^l}
E^l=\bigoplus_i \widetilde{\ker b^{0,1}_i}.
\end{equation}

The construction of the map $\psi^{k,-k}$ is also similar to the construction of $\psi_i$ in Lemma \ref{lemma: gluing bundles by transition function modulo Q}. Let $(\widetilde{f_1},\ldots,\widetilde{f_n})$ be a section of $E^l$, we define (see Equation (\ref{equation: construction of psi})) $\psi^{k,-k}$ as
\begin{equation}\label{equation: construction of psi in lemma}
\begin{split}
\psi^{k,-k}: E^l|_{U_{i_0\ldots i_k}}&\rightarrow F^{l-k}_{i_0}|_{U_{i_0\ldots i_k}}\\
(\widetilde{f_1},\ldots,\widetilde{f_n})&\mapsto \sum_j b^{k+1,-k}_{i_0\ldots i_k j}(\rho_j \widetilde{f_j}).
\end{split}
\end{equation}

In particular
\begin{equation}\label{equation: construction of psi00}
\psi^{0,0}(\widetilde{f_1},\ldots,\widetilde{f_n})=\sum_j b^{1,0}_{i j}(\rho_j \widetilde{f_j}).
\end{equation}
We know that $b^{0,1}b^{1,0}+b^{1,0}b^{0,1}=0$ hence
\begin{equation*}
\begin{split}
b^{0,1}_i\psi^{0,0}(\widetilde{f_1},\ldots,\widetilde{f_n})=& \sum_jb^{0,1}_i b^{1,0}_{i j}(\rho_j \widetilde{f_j})\\
=&-\sum_jb^{1,0}_{ij} b^{0,1}_j(\rho_j \widetilde{f_j}).
\end{split}
\end{equation*}
We know each of the $\rho_j \widetilde{f_j}$ belongs to $\ker b^{0,1}_j$ so the right hand side is zero, hence $\psi^{0,0}(\widetilde{f_1},\ldots,\widetilde{f_n})\in \ker b^{0,1}_i$.

Moreover, $\psi^{0,0}$ is surjective modulo $F^{l-1}_i$ by exactly the same argument as in Lemma \ref{lemma: gluing bundles by transition function modulo Q}.

Then we need to prove that $\psi^{k,-k}$ satisfies Condition 2. First we notice that on $U_{i_0\ldots i_k}$
\begin{equation*}
\begin{split}
(\delta\psi^{k-1,1-k})_{i_0\ldots i_k}(\widetilde{f_1},\ldots,\widetilde{f_n})=&\sum_{s=1}^{k-1} (-1)^s\psi^{k-1,1-k}_{{i_0\ldots \widehat{i_s}\ldots i_k}} (\widetilde{f_1},\ldots,\widetilde{f_n})\\
=& \sum_j \sum_{s=1}^{k-1}(-1)^s b^{k,1-k}_{{i_0\ldots \widehat{i_s}\ldots i_k}j}(\rho_j \widetilde{f_j})
\end{split}
\end{equation*}
and
\begin{equation*}
\begin{split}
(b^{p,1-p}\,\psi^{k-p,p-k})_{i_0\ldots i_k}(\widetilde{f_1},\ldots,\widetilde{f_n})=&b^{p,1-p}_ {i_0\ldots i_p} \psi^{k-p,p-k}_{{i_p\ldots i_k}} (\widetilde{f_1},\ldots,\widetilde{f_n})\\
=& \sum_j b^{p,1-p}_ {i_0\ldots i_p} b^{k-p+1,p-k}_{i_p\ldots i_k j}(\rho_j \widetilde{f_j}).
\end{split}
\end{equation*}
Moreover we have
$$
\psi^{k-1,1-k}_{i_0\ldots i_{k-1}}|_{U_{i_0\ldots i_{k}}}(\widetilde{f_1},\ldots,\widetilde{f_n})=  \sum_j  b^{k,1-k}_{i_0\ldots i_{k-1}j}(\rho_j \widetilde{f_j}).
$$

Sum them up we get
\begin{equation}\label{equation: computation of Maurer-Cartan in lemma}
\begin{split}
&(\delta \psi^{k-1,1-k}+(-1)^{k}\psi^{k-1,1-k}_{i_0\ldots i_{k-1}}|_{U_{i_0\ldots i_{k}}}+\sum_{p=0}^k b^{p,1-p}\,\psi^{k-p,p-k})(\widetilde{f_1},\ldots,\widetilde{f_n})\\
=& \sum_j\Big( \sum_{s=1}^{k}(-1)^s b^{k,1-k}_{{i_0\ldots \widehat{i_s}\ldots i_k}j}(\rho_j \widetilde{f_j})+\sum_{p=0}^k b^{p,1-p}_ {i_0\ldots i_p} b^{k-p+1,p-k}_{i_p\ldots i_k j}(\rho_j \widetilde{f_j})\Big)\\
=& \sum_j\Big( (\delta b^{k,1-k})_{i_0\ldots i_kj}(\rho_j \widetilde{f_j})+\sum_{p=0}^k b^{p,1-p}_ {i_0\ldots i_p} b^{k-p+1,p-k}_{i_p\ldots i_k j}(\rho_j \widetilde{f_j})\Big)
\end{split}
\end{equation}
Since $(F_i,b^{\bullet,1-\bullet})$ is a twisted complex, we know that on $U_{i_0\ldots i_kj}$ we have
$$
(\delta b^{k,1-k})_{i_0\ldots i_kj} +\sum_{p=0}^k b^{p,1-p}_ {i_0\ldots i_p} b^{k-p+1,p-k}_{i_p\ldots i_k j}+b^{k+1,-k}_{i_0\ldots i_kj}b^{0,1}_j=0
$$
hence the right hand side of Equation (\ref{equation: computation of Maurer-Cartan in lemma}) equals to
$$
-\sum_jb^{k+1,-k}_{i_0\ldots i_kj}b^{0,1}_j (\rho_j \widetilde{f_j}).
$$
We remember that by the definition $\rho_j \widetilde{f_j}$ is in $\ker b^{0,1}_j$ hence the above expression vanishes and the maps $\psi^{k,-k}$'s satisfy Condition 2 .
\end{proof}

\emph{The proof of Proposition} \ref{prop: descent of smooth complex on Cech cover}: We  know that each $F^{\bullet}_i$ is bounded on both directions and we have finitely many $F^{\bullet}_i$'s. Let $n$ be the largest integer such that there exists a $U_i$ such that $F^n_i\neq 0$.

Now we apply Lemma \ref{lemma: higher gluing bundles by twisted cochains} to  $F^n_i$. We notices that it automatically satisfies the condition in Lemma \ref{lemma: higher gluing bundles by twisted cochains}, as a result we can find a locally free finitely generated sheaf of $\mathcal{A}$-modules $E^n$ on $X$ together with $\mathcal{A}$-module maps
$$
\phi^{k,-k}: E^n|_{U_{i_0\ldots i_k}}\rightarrow F^{n-k}_{i_0}|_{U_{i_0\ldots i_k}}
$$
such that $\phi^0_i$ is surjective to $F^n_i$ modulo $F^{n-1}_i$ (here $b^{0,1}=0$ on $F^n_i$ so $\ker b^{0,1}_i=F^n_i$) and the $\phi^{k,-k}$'s satisfies Equation (\ref{equation: Maurer-Cartan in lemma}). We put $E^k=0$ for $k>n$.

Then we use downwards induction on the lower bound of $E^{\bullet}$: Assume we have
\begin{enumerate}
\item A cochain complex  of locally free finitely generated sheaves of $\mathcal{A}$-modules  $E^{\bullet}$ with lower bound $\geq m+1$ and upper bound $n$. Let $a^{0,1}$ denote the differentials on the $E^{\bullet}$'s.
\item For each $m+1\leq l\leq n$ and any $k\geq 0$ we have $\mathcal{A}$-module maps
$$
\phi^{k,-k}: E^l|_{U_{i_0\ldots i_k}}\rightarrow F^{l-k}_{i_0}|_{U_{i_0\ldots i_k}}
$$
which are compatible with  $b^{\bullet,1-\bullet}$ and $a^{0,1}$ in the sense of Equation (\ref{equation: phi intertwine a and b}).

\item The map induced by $\phi^{0,0}$ between cohomologies
$$
\phi^{0,0}_i: H^l(E^{\bullet}|_{U_i},a^{0,1}|_{U_i})\rightarrow H^l(F^{\bullet},b^{0,1}_i)
$$
is an isomorphism for all $l>m+1$ and is surjective for $l=m+1$.
\end{enumerate}

We want to proceed from $m+1$ to $m$. First we construct the sheaf $E^m$. For this purpose we temporarily let $E^k=0$ for all $k\leq m$ and we consider mapping cone $(G^{\bullet},c^{\bullet,1-\bullet})$ of the $\phi: \mathcal{T}(E)\to F$. According to Definition \ref{defi: mapping cone} we have
\begin{equation}
\begin{split}
G^l_i=&E^{l+1}|_{U_i}\oplus F^l_i \text{ for } l\geq m\\
G^l_i=& F^l_i \text{ for } l<m
\end{split}
\end{equation}
and
\begin{equation}\label{equation: construction of c^{k,1-k}}
\begin{split}
c^{0,1}_i=&\begin{pmatrix} a^{0,1}_i&0\\ \phi^{0,0}_i&b^{0,1}_i\end{pmatrix}\\
c^{1,0}_{ij}=&\begin{pmatrix}-id_{U_{ij}}&0\\ \phi^{1,-1}_{ij}&b^{1,0}_{ij}\end{pmatrix}\\
c^{k,1-k}_{i_0\ldots i_k}=&\begin{pmatrix}0&0\\ \phi^{k,-k}_{i_0\ldots i_k}&b^{k,1-k}_{i_0\ldots i_k}\end{pmatrix} \text{ for } k>1.
\end{split}
\end{equation}

The complex $(G^{\bullet}_i, c^{0,1}_i)$ is the ordinary mapping cone of the cochain map $ \phi^{0,0}_i: E^{\bullet}|_{U_i}\rightarrow F^{\bullet}_i$. By Assumption 3 above,  the map $\phi^{0,0}_i$ induces an isomorphism on cohomology for degree $>m+1$ and is surjective on degree $=m+1$. Hence the cohomology
$$
H^p(G^{\bullet}_i, c^{0,1}_i)=0 \text{ for } p\geq m+1.
$$
Again by Lemma \ref{lemma: higher gluing bundles by twisted cochains} there exists a locally free finitely generated sheaf of $\mathcal{A}$-modules $E^m$ together with $\mathcal{A}$-module maps
$$
\psi^{k,-k}:  E^m|_{U_{i_0\ldots i_k}}\rightarrow G^{m-k}_{i_0}|_{U_{i_0\ldots i_k}}
$$
such that
\begin{enumerate}
\item For each $i$,  the image of the map $\psi^{0,0}_i: E^m|_{U_i}\rightarrow G^m_i$ is contained in $\ker c^{0,1}_i\subset G^m_i=E^{m+1}|_{U_i}\oplus F^m_i$ and $\psi^{0,0}_i$ is onto $\ker c^{0,1}_i$ modulo $G^{m-1}_i=F^{m-1}_i$.

\item On each $U_{i_0\ldots i_k}$ the $\psi$'s satisfy the identity
\begin{equation}\label{equation: in theorem phi intertwine a and c}
(\delta \psi^{k-1,1-k})_{i_0\ldots i_k}+(-1)^{k}\psi^{k-1,1-k}_{i_0\ldots i_{k-1}}|_{U_{i_0\ldots i_{k}}}+\sum_{p=0}^k c^{p,1-p}_{i_0\ldots i_p}\,\psi^{k-p,p-k}_{i_p\ldots i_k}=0
\end{equation}
as maps from $E^m|_{U_{i_0\ldots i_k}}$ to $G^{m+1-k}_{i_0}|_{U_{i_0\ldots i_k}}$.
\end{enumerate}

Since $G^k_i=F^k_i$ for $k<m$, we can write the $\psi$'s more precisely as
\begin{equation}
\begin{split}
\psi^{0,0}_i:  &E^m|_{U_i}\rightarrow E^{m+1}|_{U_i}\oplus F^m_i\\
\psi^{k,-k}_{i_0\ldots i_k}:  &E^m|_{U_{i_0\ldots i_k}}\rightarrow F^{m-k}_{i_0}|_{U_{i_0\ldots i_k}} \text{ for } k\geq 1.
\end{split}
\end{equation}

Now let $p_E$ and $p_F$ denote the projection from $E^{m+1}|_{U_i}\oplus F^m_i$ to $E^{m+1}|_{U_i}$ and $F^m_i$ respectively, then we define $a^{0,1}: E^m\rightarrow E^{m+1}$ as
\begin{equation}
a^{0,1}|_{U_i}=p_E\circ \psi^{0,0}_i.
\end{equation}
We need to prove  that the $p_E\circ \psi^{0,0}_i$'s on different $U_i$'s actually glue together to get the map $a^{0,1}$. In fact in the proof of Lemma \ref{lemma: higher gluing bundles by twisted cochains}, Equation (\ref{equation: construction of E^l}) tells us
\begin{equation}
E^m=\bigoplus_i\widetilde{ \ker c^{0,1}_i}
\end{equation}
According to Equation (\ref{equation: construction of c^{k,1-k}}), the map $c^{0,1}_i: E^{m+1}|_{U_i}\oplus F^m_i\rightarrow E^{m+2}|_{U_i}\oplus F^{m+1}_i$ is given by
\begin{equation*}
c^{0,1}_i=\begin{pmatrix} a^{0,1}_i&0\\ \phi^{0,0}_i&b^{0,1}_i\end{pmatrix}
\end{equation*}
hence an element $g_i=(e_i,f_i)\in E^{m+1}|_{U_i}\oplus F^m_i$ is contained in $\ker c^{0,1}_i$ if and only if
$$
a^{0,1}_i(e_i)=0 \text{ and } \phi^{0,0}_i(e_i)=b^{0,1}_i(f_i)
$$
 Equation (\ref{equation: construction of psi00}) tells us for $(\widetilde{g_1},\ldots,\widetilde{g_n})=((\widetilde{e_1},\widetilde{f_1})\ldots (\widetilde{e_n},\widetilde{f_n}))\in E^m$ we have
\begin{equation*}
\psi^{0,0}_i(\widetilde{g_1},\ldots,\widetilde{g_n})=\sum_j c^{1,0}_{i j}(\rho_j \widetilde{g_j})=\sum_j c^{1,0}_{i j}(\rho_j\widetilde{e_j},\rho_j\widetilde{f_j})).
\end{equation*}
Again Equation (\ref{equation: construction of c^{k,1-k}}) tell us
$$
c^{1,0}_{ij}= \begin{pmatrix}-id_{U_{ij}}&0\\ \phi^{1,-1}_{ij}&b^{1,0}_{ij}\end{pmatrix}
$$
therefore
$$
c^{1,0}_{ij}\begin{pmatrix}\rho_j\widetilde{e_j}\\\rho_j\widetilde{f_j}\end{pmatrix}= \begin{pmatrix}-\rho_j\widetilde{e_j}|_{U_{ij}} \\ \phi^{1,-1}_{ij}(\rho_j\widetilde{e_j}|_{U_{ij}})+ b^{1,0}_{ij}(\rho_j\widetilde{f_j})\end{pmatrix}
$$
hence
\begin{equation}
\psi^{0,0}_i(\widetilde{g_1},\ldots,\widetilde{g_n})=\begin{pmatrix}\sum_j-\rho_j\widetilde{e_j}|_{U_{ij}} \\ \sum_j\phi^{1,-1}_{ij}(\rho_j\widetilde{e_j}|_{U_{ij}})+ b^{1,0}_{ij}(\rho_j\widetilde{f_j})\end{pmatrix}.
\end{equation}
As a result
\begin{equation}
(a^{0,1}|_{U_i})(\widetilde{g_1},\ldots,\widetilde{g_n})=p_E( \psi^{0,0}_i(\widetilde{g_1},\ldots,\widetilde{g_n}))=-\sum_j\rho_j\widetilde{e_j}|_{U_{ij}}
\end{equation}
therefore the $a^{0,1}$ on different $U_i$'s can glue together to an $\mathcal{A}$-module map $a^{0,1}: E^m\rightarrow E^{m+1}$. Moreover, its image $-\sum\rho_j\widetilde{e_j}$ is in the kernel of $a^{0,1}: E^{m+1}\rightarrow E^{m+2}$ hence we have extended the complex $(E^{\bullet}, a^{0,1})$ to degree $=m$.

Next we define $\phi^{k,-k}_{i_0\ldots i_k}: E^m|_{U_{i_0\ldots i_k}}\rightarrow F^{m-k}_{i_0}|_{U_{i_0\ldots i_k}}$ as
\begin{equation}
\begin{split}
\phi^{0,0}_i = p_F\circ \psi^{0,0}_i:& E^m|_{U_i}\rightarrow   F^m_i\\
\phi^{k,-k}_{i_0\ldots i_k}= \psi^{k,-k}_{i_0\ldots i_k}:& E^m|_{U_{i_0\ldots i_k}}\rightarrow F^{m-k}_{i_0}|_{U_{i_0\ldots i_k}} \text{ for } k\geq 1.
\end{split}
\end{equation}

Recall that Equation (\ref{equation: in theorem phi intertwine a and c}) tells us
$$
(\delta \psi^{k-1,1-k})_{i_0\ldots i_k}+(-1)^{k}\psi^{k-1,1-k}_{i_0\ldots i_{k-1}}|_{U_{i_0\ldots i_{k}}}+\sum_{p=0}^k c^{p,1-p}_{i_0\ldots i_p}\,\psi^{k-p,p-k}_{i_p\ldots i_k}=0
$$
as maps from $E^m|_{U_{i_0\ldots i_k}}$ to $G^{m+1-k}_{i_0}|_{U_{i_0\ldots i_k}}$ and we need to show that the above identity implies
\begin{equation}\label{equation: phi intertwine a and b in the theorem}
( \delta \phi^{k-1,1-k})_{i_0\ldots i_k}+\sum_{p=0}^k b^{p,1-p}_{i_0\ldots i_p}\phi^{k-p,p-k}_{i_p\ldots i_k}-\phi^{k,-k}_{i_0\ldots i_k} a^{0,1}_{i_k}-\phi^{k-1,1-k}_{i_0\ldots i_{k-1}}|_{U_{i_0\ldots i_k}}=0.
\end{equation}
Let us check it for $k=1$.
From the definition of the $c^{k,1-k}$'s as in Equation (\ref{equation: construction of c^{k,1-k}}) we can easily deduce that
$$
c^{k,1-k}\psi^{0,0}=\begin{cases}\phi^{k,-k}a^{0,1}+b^{k,1-k}\phi^{0,0} &\text{ if } k\geq 2 \\
(-a^{0,1},  \phi^{1,-1}a^{0,1}+b^{1,0}\phi^{0,0}) &\text{ if } k=1\\
\phi^{0,0}a^{0,1}+b^{0,1}\phi^{0,0}&\text{ if } k=0
\end{cases}
$$
In the $k=1$ case we know that $ \delta \psi^{0,0}$ and $ \delta \phi^{0,0}$ are automatically zero hence Equation (\ref{equation: in theorem phi intertwine a and c}) becomes
\begin{equation*}
\begin{split}
&c^{0,1}_{i_0}\psi^{1,-1}_{i_0i_1}+c^{1,0}_{i_0}\psi^{0,0}_{i_0i_1}-\psi^{0,0}_{i_0}|_{U_{i_0i_1}}\\
=& (0, b^{0,1}_{i_0}\phi^{1,-1}_{i_0i_1})+(-a^{0,1}, b^{1,0}_{i_0i_1}\phi^{0,0}_{i_1}+\phi^{1,-1}_{i_0i_1}a^{0,1})-(-a^{0,1},\phi^{0,0}_{i_0}|_{U_{i_0i_1}})\\
=& (0, b^{0,1}_{i_0}\phi^{1,-1}_{i_0i_1}+b^{1,0}_{i_0i_1}\phi^{0,0}_{i_1}+\phi^{1,-1}_{i_0i_1}a^{0,1}-\phi^{0,0}_{i_0}|_{U_{i_0i_1}})=0.
\end{split}
\end{equation*}
Therefore
$$
b^{0,1}_{i_0}\phi^{1,-1}_{i_0i_1}+b^{1,0}_{i_0i_1}\phi^{0,0}_{i_1}+\phi^{1,-1}_{i_0i_1}a^{0,1}-\phi^{0,0}_{i_0}|_{U_{i_0i_1}}=0
$$
which is exactly Equation (\ref{equation: phi intertwine a and b in the theorem}) in the case $k=1$. The $k=0$ and $k \geq 2$ cases are similar.

We also need to show that after we introduce $E^m$, the map
$$
\phi^{0,0}_i: H^{m+1}(E^{\bullet}|_{U_i},a^{0,1}|_{U_i})\rightarrow H^{m+1}(F^{\bullet},b^{0,1}_i)
$$
becomes an isomorphism and
$$
\phi^{0,0}_i: H^{m}(E^{\bullet}|_{U_i},a^{0,1}|_{U_i})\rightarrow H^{m}(F^{\bullet},b^{0,1}_i)
$$
is surjective.

We already know that in degree $m+1$ the induced map is surjective. Moreover by Lemma \ref{lemma: higher gluing bundles by twisted cochains} we know that $\psi^{0,0}_i: E^m\rightarrow \ker c^{0,1}_i \subset E^{m+1}\oplus F^m_i$ is surjective. In more details if $(e^{m+1},f^m)\in E^{m+1}\oplus F^m_i$ such that
$$
a^{0,1}_i(e^{m+1})=0 \text{ and } \phi^{0,0}_i(e^{m+1})=b^{0,1}_i(f^m)
$$
then there exists an $e^m$ and an $f^{m-1}$ such that $\psi^{0,0}_i(e^m)=(e^{m+1},f^m)+(0,b^{0,1}_i(f^{m-1}))$. We remember that by definition $a^{0,1}=p_E \psi^{0,0}_i$ and $\phi^{0,0}_i=p_F\psi^{0,0}_i$ hence we get
\begin{equation}\label{equation: a(e)=e in the theorem}
a^{0,1}e^m=e^{m+1} \phi^{0,0}_i(e^m)=f^m+b^{0,1}_i(f^{m-1}).
\end{equation}
and
\begin{equation}\label{equation: phi(e)=f in the theorem}
a^{0,1}e^m=e^{m+1} \phi^{0,0}_i(e^m)=f^m+b^{0,1}_i(f^{m-1}).
\end{equation}
Equation (\ref{equation: a(e)=e in the theorem}) implies
$$
\phi^{0,0}_i: H^{m+1}(E^{\bullet}|_{U_i},a^{0,1}|_{U_i})\rightarrow H^{m+1}(F^{\bullet},b^{0,1}_i)
$$
is also injective and hence an isomorphism. To prove $\phi^{0,0}_i: H^{m}(E^{\bullet}|_{U_i},a^{0,1}|_{U_i})\rightarrow H^{m}(F^{\bullet},b^{0,1}_i)$ is surjective,  we consider the special case that $e^{m+1}=0$ and $b^{0,1}_i(f^m)=\phi^{0,0}_i(e^{m+1})=0$. In this case Equation (\ref{equation: phi(e)=f in the theorem}) tells us the map
$$
\phi^{0,0}_i: H^{m}(E^{\bullet}|_{U_i},a^{0,1}|_{U_i})\rightarrow H^{m}(F^{\bullet},b^{0,1}_i)
$$
is surjective. Now we have proved that $(E^{\bullet},a^{0,1})$ satisfies the induction assumption 1, 2 and 3 at degree $m$.

Lastly we need to show that the induction steps eventually stops at some degree. Let $m_0$ be the smallest integer such that $F^{m_0}_i\neq 0$ for some $i$. By downward induction we get a complex $(E^{l},a^{0,1}), ~m_0+1\leq l\leq n$ and maps $\phi^{k,-k}: E\rightarrow F$ which is compatible with the $a$'s and $b$'s and
$$
\phi^{0,0}_i: H^{l}(E^{\bullet}|_{U_i},a^{0,1}|_{U_i})\rightarrow H^{l}(F^{\bullet},b^{0,1}_i)
$$
is an isomorphism for $l>m_0+1$ and is surjective for $l=m_0+1$.

Now we need to construct $E^{m_0}$ and $\psi^{0,0}_i: E^{m_0}|_{U_i}\rightarrow E^{m_0+1}|_{U_i}\oplus F^{m_0}_i$ in a slightly different way. For this we use the map
$$
c^{0,1}_i: E^{m_0+1}|_{U_i}\oplus F^{m_0}_i\rightarrow   E^{m_0+2}|_{U_i}\oplus F^{m_0+1}_i
$$
and we get $\ker c^{0,1}_i $ which is a locally free finitely generated sheaf of $\mathcal{A}$-modules on $U_i$ by  Corollary \ref{coro: kernel of acyclic above complex is a vector bundle}.

Now we want to show that $\ker c^{0,1}_i $ can glue together to get a locally free finitely generated sheaf of $\mathcal{A}$-modules on $X$. In fact the map $c^{1,0}_{ij}$'s satisfy $c^{1,0}_{ik}-c^{1,0}_{ij}c^{1,0}_{jk}+c^{2,-1}_{ijk}c^{0,1}_k-c^{0,1}_ic^{2,-1}_{ijk}=0$. Now we are considering $\ker c^{0,1}_j$ and $F^{m_0-1}_i=0$ hence the last two term vanish and we have
$$
c^{1,0}_{ik}-c^{1,0}_{ij}c^{1,0}_{jk}=0 \text{ on } \ker c^{0,1}_k.
$$
In other words the maps $c^{1,0}_{ij}$'s give us transition function from $\ker c^{0,1}_j$ to $\ker c^{0,1}_i$ and they satisfy the cocycle condition.

Therefore we can define the resulting locally free finitely generated sheaf of $\mathcal{A}$-modules by $E^{m_0}$. In particular we have
$$
E^{m_0}|_{U_i}=\ker c^{0,1}_i\subset E^{m_0+1}|_{U_i}\oplus F^{m_0}_i.
$$
Moreover we define $a^{0,1}: E^{m_0}\rightarrow E^{m_0+1}$ and $\phi^{0,0}_i: E^{m_0}|_{U_i}\rightarrow F^{m_0}_i$ be the projection map onto the first and second component respectively, and the $\phi^{k,-k}$'s are all zero for $k \geq 1$. It is easy to see that $\phi^{0,0}_i$ induces an isomorphism from $H^{m_0}(E^{\bullet}|_{U_i},a^{0,1})$ to $H^{m_0}(F^{\bullet}_i,b^{0,1})$ and we finished the proof.

$\square$

\begin{rmk}
The proof of Proposition \ref{prop: descent of smooth complex on Cech cover} is inspired by the proof of Lemma 1.9.5 in \cite{thomason1990higher}.
\end{rmk}

\subsection{The quasi-fully faithfulness of $\mathcal{T}$}\label{subsection: quasi-fully faithful}
In this subsection we prove that $\mathcal{T}$ is quasi-fully faithful. The dg-functor
 $$
\mathcal{T}:~\mathfrak{L}(X,\mathcal{A})\rightarrow \text{Tw}_{\text{perf}}(X)
 $$
 gives a map between complexes of morphisms
  \begin{equation}\label{equation: restriction of morphisms}
\mathcal{T}:~\mathfrak{L}(X,\mathcal{A})(E^{\bullet}, F^{\bullet})\rightarrow \text{Tw}_{\text{perf}}(X)(\mathcal{T}(E^{\bullet}),\mathcal{T}(F^{\bullet}))
 \end{equation}
where for a morphism $\varphi: E^{\bullet}\rightarrow F^{\bullet}$ of degree $n$, its image $\mathcal{T}(\varphi)$ is an element of total degree $n$ in $C^{\bullet}(\mathcal{U}, \text{Hom}^{\bullet}(\mathcal{T}(E^{\bullet}),\mathcal{T}(F^{\bullet})))$ given by
$$
\mathcal{T}(\varphi)^{0,n}_i=\varphi|_{U_i}
$$
and
$$
\mathcal{T}(\varphi)^{k,n-k}=0,~\forall~k\geq 1.
$$
\begin{prop}[Quasi-fully faithfulness]\label{prop: descent of morphisms to twisted complex}
As in Proposition \ref{prop: descent of smooth complex on Cech cover}, let $X$ is a paracompact space with a soft structure sheaf $\mathcal{A}$. Moreover if the open cover $U_i$ is finite and good, then the  morphism $\mathcal{T}$ in Equation (\ref{equation: restriction of morphisms}) is a quasi-isomorphism of complexes.
\end{prop}


The proof of Proposition \ref{prop: descent of morphisms to twisted complex}

First we need to get a more explicit description of the differentials in the complex Tw$(X)(\mathcal{T}(E^{\bullet}),\mathcal{T}(F^{\bullet}))$.

\begin{lemma}
Let $\phi^{p,q}$ be an element in $\text{Tw}^{p,q}(X)(\mathcal{T}(E^{\bullet}),\mathcal{T}(F^{\bullet}))$.
We can write $d\phi^{p,q}$ as
\begin{equation}\label{equation: decompose of d on morphisms}
d\phi^{p,q}=\hat{\delta}\phi^{p,q}+(-1)^p \hat{d}\phi^{p,q}
\end{equation}
where $\hat{\delta}\phi^{p,q}$ is of degree $(p+1,q)$ and is given by
$$
(\hat{\delta}\phi^{p,q})_{i_0i_1\ldots i_{p+1}}=\sum_{k=0}^{p+1}(-1)^k \phi^{p,q}_{i_0\ldots \widehat{i_k} \ldots i_{p+1}}|_{U_{i_0i_1\ldots i_{p+1}}}
$$
and $\hat{d}\phi^{p,q}$ is of degree $(p,q+1)$ and is given by
$$
\hat{d}\phi^{p,q}=d_F \circ\phi^{p,q}-(-1)^q \phi^{p,q}\circ d_E.
$$
\end{lemma}
\begin{proof}
We know that in general the differential is given in Equation (\ref{equation: differential on morphisms of twisted complexes}) by
$$
d \phi=\delta \phi+b\cdot \phi-(-1)^{|\phi|}\phi\cdot a.
$$

Now by the very definition of $\mathcal{T}(E^{\bullet})$ and $\mathcal{T}(F^{\bullet})$ we know that
$$
a^{0,1}_i=d_E|_{U_i},~a^{1,0}_{ij}=id_E|_{U_{ij}}, ~a^{k,1-k}=0, ~ \forall k\geq 2
$$
and
$$
b^{0,1}_i=d_F|_{U_i},~b^{1,0}_{ij}=id_F|_{U_{ij}}, ~b^{k,1-k}=0, ~ \forall k\geq 2
$$
therefore
$$
d\phi^{p,q}=\delta \phi^{p,q}+b^{0,1}\cdot \phi^{p,q}+b^{1,0}\cdot \phi^{p,q}-(-1)^{p+q}\phi^{p,q}\cdot a^{0,1}-(-1)^{p+q}\phi^{p,q}\cdot a^{1,0}
$$
and we get the result we want.
\end{proof}

Then we can proceed to the proof. First we have the following proposition

\begin{prop}\label{prop: surjective of morphisms}
The map $\mathcal{T}:~\mathfrak{L}(X,\mathcal{A})(E^{\bullet}, F^{\bullet})\rightarrow \text{Tw}_{\text{perf}}(X)(\mathcal{T}(E^{\bullet}),\mathcal{T}(F^{\bullet}))$ is surjective on the level of cohomology. In other words, for any closed element $\phi$ in Tw$(X)(\mathcal{T}(E^{\bullet}),\mathcal{T}(F^{\bullet}))$, there exists a closed element $\widetilde{\phi}$ in $\mathfrak{L}(X,\mathcal{A})(E^{\bullet},F^{\bullet})$ and an element $\varphi$ in Tw$(X)(\mathcal{T}(E^{\bullet}),\mathcal{T}(F^{\bullet}))$ such that
$$
\phi-\mathcal{T}(\widetilde{\phi})=d \varphi.
$$
\end{prop}
\begin{proof}
Let $n$ be the total degree of $\phi$. By Equation (\ref{equation: decompose of d on morphisms}) we have
$$
d\phi^{k,n-k}=\hat{\delta}\phi^{k,n-k}+(-1)^k \hat{d}\phi^{k,n-k},~\forall ~k\geq 0.
$$

 $\phi$ is closed implies that
\begin{equation}\label{equation: closed condition on morphism}
\begin{split}
\hat{\delta}\phi^{k,n-k}+(-1)^{k+1}\hat{d}\phi^{k+1,n-k-1}=&0,~\forall ~k\geq 0\\
\hat{d}\phi^{0,n}=&0.
\end{split}
\end{equation}

We know that $\phi^{0,n}$ consists of pieces $\phi^{0,n}_i$ on $U_i$. Using the partition of unity $\rho_i$ we define $\widetilde{\phi}$ as
\begin{equation}
\widetilde{\phi}:=\sum_i \rho_i \phi^{0,n}_i.
\end{equation}
It is obvious that $\widetilde{\phi}$ is a globally defined degree $n$ map from $E^{\bullet}$ to  $F^{\bullet}$. Moreover $\widetilde{\phi}$ is closed since we have $\hat{d}\phi^{0,n}=0$ hence
$$
d_F\circ \widetilde{\phi}-(-1)^n \widetilde{\phi}\circ d_E=\sum_i \rho_i \hat{d}\phi^{0,n}_i=0.
$$.

We still need to show that  $\mathcal{T}(\widetilde{\phi})$ and $\phi$ are cohomologous. For this we define $\varphi$ with total degree $n-1$ by
\begin{equation}\label{equation: definition of varphi in morphisms}
\varphi^{k,n-1-k}_{i_0\ldots i_k}:=\sum_j \rho_j \phi^{k+1,n-1-k}_{ji_0\ldots i_k }.
\end{equation}

Let us check that $\phi-\mathcal{T}(\widetilde{\phi})=d \varphi$ in each degree. First in degree $(0,n)$
\begin{equation*}
\begin{split}
[\phi-\mathcal{T}(\widetilde{\phi})]^{0,n}_i=& \phi^{0,n}_i-\sum_j\rho_j \phi^{0,n}_j|_{U_i}\\
=&\sum_j \rho_j \phi^{0,n}_i-\sum_j \rho_j \phi^{0,n}_j|_{U_i}\\
=&\sum_j \rho_j (\hat{\delta}\phi^{0,n})_{ji}.
\end{split}
\end{equation*}
Equation (\ref{equation: closed condition on morphism}) tells us that $(\hat{\delta}\phi^{0,n})_{ji}=\hat{d}\phi^{1,n-1}_{ji}$ therefore
\begin{equation*}
\begin{split}
[\phi-\mathcal{T}(\widetilde{\phi})]^{0,n}_i=&\sum_j \rho_j \hat{d}\phi^{1,n-1}_{ji}\\
=&\hat{d}\sum_j \rho_j \phi^{1,n-1}_{ji}\\
=&\hat{d}\varphi^{0,n-1}_i.
\end{split}
\end{equation*}

For degree $(k,n-k)$, $k \geq 1$, since $\mathcal{T}(\widetilde{\phi})^{k,n-k}=0$ it is clear that
$$
[\phi-\mathcal{T}(\widetilde{\phi})]^{k,n-k}_{i_0\ldots i_k}=\phi^{k,n-k}_{i_0\ldots i_k}.
$$
On the other hand  Equation (\ref{equation: decompose of d on morphisms}) tells us
$$
(d\varphi)^{k,n-k}_{i_0\ldots i_k}=(\hat{\delta}\varphi^{k-1,n-k})_{i_0\ldots i_k}+(-1)^k (\hat{d}\varphi^{k,n-k-1})_{i_0\ldots i_k}
$$
where
\begin{equation*}
\begin{split}
(\hat{\delta}\varphi^{k-1,n-k})_{i_0\ldots i_k}=&\sum_{s=0}^k(-1)^s \varphi^{k-1,n-k}_{i_0\ldots \widehat{i_s}\ldots i_k}\\
=&\sum_{s=0}^k(-1)^s \sum_j \rho_j \phi^{k,n-k}_{ji_0\ldots \widehat{i_s}\ldots i_k}\\
=& \sum_j \rho_j \sum_{s=0}^k(-1)^s \phi^{k,n-k}_{ji_0\ldots \widehat{i_s}\ldots i_k}\\
=& \sum_j \rho_j [-(\hat{\delta}\phi^{k,n-k})_{ji_0\ldots i_k }+\phi^{k,n-k}_{i_0\ldots i_k}]\\
=& -\sum_j \rho_j (\hat{\delta}\phi^{k,n-k})_{j i_0\ldots i_k }+ \sum_j \rho_j \phi^{k,n-k}_{i_0\ldots i_k}\\
=& -\sum_j \rho_j (\hat{\delta}\phi^{k,n-k})_{j i_0\ldots i_k }+ \phi^{k,n-k}_{i_0\ldots i_k}
\end{split}
\end{equation*}
and
$$
\hat{d}\varphi^{k,n-k-1}_{i_0\ldots i_k}=\sum_j\rho_j \hat{d}\phi^{k+1,n-k-1}_{ji_0\ldots i_k}.
$$
Recall that $\phi$ is closed hence $\hat{\delta}\phi^{k,n-k}+(-1)^k\hat{d}\phi^{k+1,n-k-1}=0$, therefore we get cancellations and
$$
(d\varphi)^{k,n-k}_{i_0\ldots i_k}=  \phi^{k,n-k}_{i_0\ldots i_k}
$$
hence we get
$$
[\phi-\mathcal{T}(\widetilde{\phi})]^{k,n-k}_{i_0\ldots i_k}=(d\varphi)^{k,n-k}_{i_0\ldots i_k}.
$$
\end{proof}

Next we have the following proposition, which gives the injectivity of $\mathcal{T}$.

\begin{prop}\label{prop: injective of morphisms}
If $\mathcal{T}(\phi)$ is a coboundary in Tw$(X)(\mathcal{T}(E^{\bullet}),\mathcal{T}(F^{\bullet}))$, then $\phi$ is a coboundary in $\mathfrak{L}(X,\mathcal{A})(E^{\bullet},F^{\bullet})$.
\end{prop}
\begin{proof}
Let the total degree of $\phi$ be $n$. Let $\varphi$ be an element with total degree $n-1$ in  Tw$(X)(\mathcal{T}(E^{\bullet}),\mathcal{T}(F^{\bullet}))$ such that
$$
\mathcal{T}(\phi)=d\varphi.
$$

Define $\widetilde{\varphi}$ in $\mathfrak{L}(X,\mathcal{A})(E^{\bullet},F^{\bullet})$ as
$$
\widetilde{\varphi}=\sum_j\rho_j \varphi^{0,n-1}_j
$$

Similar to the proof of Proposition  \ref{prop: descent of morphisms to twisted complex}, we could check that
$$
\phi=d \widetilde{\varphi}.
$$
\end{proof}

Proposition \ref{prop: surjective of morphisms} and Proposition \ref{prop: injective of morphisms} together give Proposition \ref{prop: descent of morphisms to twisted complex}.

\section{Twisted perfect complex and homotopy limit}\label{section: twisted complex and homotopy limit}
In this section we consider $\mathfrak{L}$ as a presheaf of dg-categories: for any open subset $U\subset X$, $\mathfrak{L}(U)$ is the dg-category of bounded complexes of locally free sheaves on $U$.  For an open cover $\mathcal{U}=\{U_i\}$ of $X$, its \emph{\v{C}ech nerve} is a simplicial space
$$
\begin{tikzcd}
\cdots  \arrow[yshift=1.2ex]{r}\arrow[yshift=0.4ex]{r}\arrow[yshift=-0.4ex]{r}\arrow[yshift=-1.2ex]{r}& \coprod U_i\cap U_j\cap U_k \arrow[yshift=1ex]{r}\arrow{r}\arrow[yshift=-1ex]{r}& \coprod U_i\cap U_j  \arrow[yshift=0.7ex]{r}\arrow[yshift=-0.7ex]{r}  &  \coprod U_i.
\end{tikzcd}
$$
We consider the locally free sheaves on the spaces and get the following cosimplicial diagram of dg-categories
\begin{equation}\label{equation: cosimplicial diagram of L of an open cover}
\begin{tikzcd}
\prod \mathfrak{L}(U_i) \arrow[yshift=0.7ex]{r}\arrow[yshift=-0.7ex]{r}& \prod \mathfrak{L}(U_i\cap U_j) \arrow[yshift=1ex]{r}\arrow{r}\arrow[yshift=-1ex]{r}  &  \prod \mathfrak{L}(U_i\cap U_j\cap U_k)  \arrow[yshift=1.2ex]{r}\arrow[yshift=0.4ex]{r}\arrow[yshift=-0.4ex]{r}\arrow[yshift=-1.2ex]{r}&\cdots
\end{tikzcd}
\end{equation}

In general descent theory, the descent data of $\mathfrak{L}(U_i)$ is given by the \emph{homotopy limit} of the cosimplicial diagram (\ref{equation: cosimplicial diagram of L of an open cover}).

In our case we have an explicit construction of the homotopy limit.
\begin{thm}\label{thm: twisted complex is homotopy limit}
If the open cover $\{U_i\}$ of $X$ is good, then the dg-category of twisted complexes on $X$, $\text{Tw}_{\text{perf}}(X)$ is quasi-equivalent to the homotopy limit of  $\mathfrak{L}(U_I)$.
\end{thm}
\begin{proof}
See \cite{block2017explicit}.
\end{proof}

With Theorem \ref{thm: twisted complex is homotopy limit}, we can interpret the result of Theorem \ref{thm: quasi-equivalence of bundles and twisted complexes} as  follows  \begin{thm}\label{thm: locally free sheaves is weakly equivalent to the homotopy limit in section}
 Let $X$ is a paracompact space with a soft structure sheaf $\mathcal{A}$. Moreover if the open cover $U_i$ is finite and good, then we have a quasi-equivalence between dg-categories
 $$
 \mathfrak{L}(X)\overset{\sim}{\to}\text{holim}_{I}\mathfrak{L}(U_I)
 $$
 \end{thm}

At the end of this paper we should point out that it is important to generalize the result in Theorem \ref{thm: quasi-equivalence of bundles and twisted complexes} from covers to \emph{hypercovers}.

\begin{conj}\label{conj: descent for hypercovers}
Let $(X,\mathcal{A})$ be a compact ringed space with soft structure sheaf $\mathcal{A}$. Then for any hypercover $U_{\bullet}$ of $X$, the natural functor
$$
\mathfrak{L}(X) \rightarrow  \text{ho}\lim_{U_{\bullet}} \mathfrak{L}(U_{\bullet})
$$
is a quasi-equivalence.
\end{conj}

Let us briefly talk about the significance of Conjecture \ref{conj: descent for hypercovers}. 
We consider dgCat-$Pr(X)$, the category of  presheaves valued in  dg-categories and we have the  local model structure on dgCat-$Pr(X)$, see \cite{holstein2014morita}. In this case we have the following result.

\begin{prop}[Lemma 2.9 in \cite{holstein2014morita}]\label{prop: fibrant obj in dg sheaves and descent}
If $\mathcal{F}\in \text{dgCat-}Pr(X)$ is both objectwise fibrant and satisfies descent for any hypercover, then $\mathcal{F}$ is an fibrant object for the local model structure on dgCat-$Pr(X)$.
\end{prop}

\begin{rmk}
According to Remark 2.13 in \cite{holstein2014morita}, we do not have the "if and only if" result in Proposition \ref{prop: fibrant obj in dg sheaves and descent} mainly because the category dgCat does not have a symmetric monoidal model structure.
\end{rmk}

With Proposition \ref{prop: fibrant obj in dg sheaves and descent} in hand, if Conjecture \ref{conj: descent for hypercovers} is true, then the presheaf $\mathfrak{L}$ is a fibrant object in the category of presheaves of dg-categories.

\appendix
\section{Some generalities of soft sheaves}\label{appendix: some generalities of soft sheaves}
We collect here some results in sheaf theory which is necessary for our use, for reference see \cite{bredon1997sheaf} Chapter I and II.

\begin{defi}\label{defi: soft sheaf}
A sheaf $\mathcal{F}$ on a topological space $X$ is called \emph{soft} if any section over any closed subset of $X$ can be extended to a global section. In other words, for any closed subset $K\subset X$, the restriction map $\mathcal{F}(X)\rightarrow \mathcal{F}(K) $ is surjective.
\end{defi}

We also have the  concept of fine sheaf, which is related to soft sheaf.

\begin{defi}\label{defi: fine sheaf}
Let $X$ be a paracompact space. A sheaf $\mathcal{F}$ of groups over $X$ is \emph{fine} if for every two disjoint closed subsets $A,B\subset X, A\cap B=\emptyset $, there is an endomorphism of the sheaf of groups $\mathcal{F}\rightarrow \mathcal{F}$  which restricts to the identity in a neighborhood of $A$ and to the $0$ endomorphism in a neighborhood of $B$.
\end{defi}

Every fine sheaf is soft but in general not every soft sheaf is fine. However we have the following proposition:

\begin{prop}\label{prop: sheaf of modules of soft rings is fine}
Let $X$ be a paracompact space and $\mathcal{A}$ a soft sheaf of rings with unit on $X$. Then any sheaf of $\mathcal{A}$-modules is fine. In particular $\mathcal{A}$-itself is fine.
\end{prop}
\begin{proof}
See \cite{bredon1997sheaf} Theorem 9.16.
\end{proof}

\begin{eg}\label{eg: fine sheaves}
Let $X$ be a paracompact Hausdorff topological space, then the sheaf of continuous functions on $X$ is soft and hence fine.

Moreover, let $X$ be a Hausdorff smooth manifold, then the sheaf of $C^{\infty}$-functions on  $X$ is soft and hence fine.
\end{eg}

We have the following properties for soft sheaf of rings:

\begin{lemma}\label{lemma: partition of unity}
Let $(X, \mathcal{A})$ be a paracompact space such that the structure sheaf $\mathcal{A}$ is soft, then for any locally finite open cover $\mathcal{U}=\{U_i\}$ of $X$, there exists a partition of unity subordinate to the open cover $U_i$, i.e. there exists sections $\rho_i$ of $\mathcal{A}$ such that $\text{supp}\rho_i\subset U_i$ and $\sum_i \rho_i=1$.
\end{lemma}
\begin{proof}
We know that $\mathcal{A}$ is soft hence fine. Therefore it is obvious.
\end{proof}

\begin{lemma}\label{lemma: extension of a locally free finitely generated sheaf}
Let $(X, \mathcal{A})$ be a paracompact space such that the structure sheaf $\mathcal{A}$ is soft. Let $U$ be a contractible open subset of $X$, then any finite rank locally free sheaf on $U$ could be extended to a finite rank locally free sheaf on $X$.
\end{lemma}
\begin{proof}
Since $U$ is contractible, we know that any finite rank locally free sheaf on $U$ is trivializable and can be extended to $X$.
\end{proof}

We also have the following useful lemma which is not limited to modules of soft sheaves:

\begin{lemma}\label{lemma: kernel of surjective locally free sheaf map is still locally free}
Let $(X, \mathcal{A})$ be a paracompact space such that the structure sheaf $\mathcal{A}$ is soft, $\mathcal{E}$ and $\mathcal{F}$ are two locally free finite generated sheaves of $\mathcal{A}$-modules  and $\pi$ is an $\mathcal{A}$-module map between them. If $\pi$ is surjective, then $\ker \pi$ is  a locally free finite generated  sub-sheaf of $\mathcal{E}$.
\end{lemma}
\begin{proof}
It is a standard result in sheaf theory.
\end{proof}

\begin{coro}\label{coro: kernel of acyclic above complex is a vector bundle}
Let $(X, \mathcal{A})$ be a paracompact space such that the structure sheaf $\mathcal{A}$ is soft. Let $(\mathcal{F}^{\bullet},d^{\bullet})$ be a bounded above cochain complex of locally free finite generated sheaves of $\mathcal{A}$-modules on $X$. If $l$ is an integer such that the cohomology $H^n(\mathcal{F}^{\bullet},d^{\bullet})=0$ for all $n>l$, then for any $n\geq l$, $\ker d^n$ is a locally free finitely generated sub-sheaf of $\mathcal{F}^n$.
\end{coro}
\begin{proof}
Since $\mathcal{F}^{\bullet}$ is bounded above, it could be easily proved by downward induction and repeatedly applying Lemma \ref{lemma: kernel of surjective locally free sheaf map is still locally free}.
\end{proof}


\bibliography{Descentofcohbib}{}
\bibliographystyle{alpha}

\end{document}